\documentclass{amsart}
\usepackage{amsmath}
\usepackage{amsfonts}
\usepackage{amsthm, upref}
\usepackage{graphicx}
\usepackage[usenames, dvipsnames]{color}

\newcommand{\comment}[1]{}
\newcommand{\eq}{\begin{equation}}
\newcommand{\en}{\end{equation}}
\newcommand{\rr}{\mathbb{R}}
\newcommand{\norm}[1]{\left\lVert #1 \right\rVert}
\newcommand{\abs}[1]{\left\lvert #1 \right\rvert}

\newcommand{\mcal}[1]{\mathcal{#1}}
\newcommand{\iprod}[1]{\left\langle #1 \right\rangle }

\newcommand{\wass}{\mathcal{W}}
\newcommand{\ent}{H}
\newcommand{\met}{\mathcal{X}}
\newcommand{\fb}{\mathcal{F}}

\newcommand{\lip}{\mathcal{L}}
\newcommand{\ctilde}{\widetilde C}

\newcommand{\xbar}{\overline{X}}

\newcommand{\diam}{\text{diam}}

\begin{document}

\theoremstyle{plain}
\newtheorem{thm}{Theorem}
\newtheorem{lemma}[thm]{Lemma}
\newtheorem{prop}[thm]{Proposition}
\newtheorem{cor}[thm]{Corollary}

\theoremstyle{definition}
\newtheorem{defn}{Definition}
\newtheorem{asmp}{Assumption}
\newtheorem{notn}{Notation}
\newtheorem{prb}{Problem}

\theoremstyle{remark}
\newtheorem{rmk}{Remark}
\newtheorem{exm}{Example}
\newtheorem{clm}{Claim}


\title[Concentration of diffusions]{Concentration for multidimensional diffusions and their boundary local times}

\author{Soumik Pal}
\address{Department of Mathematics\\ University of Washington\\ Seattle, WA 98195}
\email{soumik@u.washington.edu}

\thanks{This research is partially supported by NSF grant DMS-1007563}




\date{\today}

\begin{abstract} 
We prove that probability laws of certain multidimensional semimartingales which includes time-inhomogenous diffusions, under suitable assumptions, satisfy Quadratic Transportation Cost Inequality under the uniform metric. From this we derive concentration properties of Lipschitz functions of process paths that depend on the entire history. In particular, we estimate concentration of boundary local time of reflected Brownian motions on a polyhedral domain. We work out explicit applications of consequences of measure concentration for the case of Brownian motion with rank-based drifts. 
\end{abstract}

\maketitle


\section{Introduction} Consider the sample space $(\Omega, \mathcal{F})$ where $\Omega$ is a metric space and $\mcal{F}$ is the associated Borel $\sigma$-algebra. We say that a probability measure $\mu$ on $(\Omega, \mathcal{F})$ has the measure concentration property if the following statement holds. For any set $A\in \mathcal{F}$ such that $\mu(A) \ge 1/2$, one has $\mu(A_r)$ \textit{very close} to one. Here $\mu(A_r)$ refers to the set of all points that are at a distance no larger than $r$ from $A$. The closeness is typically expressed as a gaussian tail estimate in $r$. Moreover, for fixed $r$, the probability $\mu(A_r)$ tends to one exponentially fast in the dimension of the underlying space. Concentration inequalities and their applications have become an integral part of modern probability theory. See, for example, the seminal articles by Talagrand \cite{T91, T94, T95, T96, T96b}. An excellent account can be found in the monograph \cite{L} by Ledoux to which we refer the reader for a survey of the (pre-2001) literature.

Throughout this article our sample space is going to be some subset of the space of continuous function on $[0,\infty)$ denoted by $C[0,\infty)$ and products of such spaces. A sample path is denoted by $\omega(t)$, $0\le t < \infty$. The filtration will be the natural filtration made right-continuous and suitably completed under the probability measures we consider. The metric on the sample space will be mostly given by the uniform metric: $d(\omega, \omega')=\sup_t \abs{\omega(t) - \omega'(t)}$.

The probability measures we consider on the above sample space are the laws of multidimensional semimartingales including diffusions. The reason for considering diffusion laws on the \textit{path space} has some strong motivation from applications. For example, consider the boundary local time of a reflected Brownian motion in an infinite wedge. Despite its significance in several areas of probability including queueing theory and mathematical finance (see, for example, the excellent survey by Fernholz and Karatzas \cite{FK} for applications to Stochastic Portfolio Theory), very little is known about such local times. The challenge is the fact that such local times are functions of the entire history of the path (as opposed to being functions of the one dimensional marginals). We show how to estimate the fluctuation of such local times and explicitly compute the case of Brownian motion with constant drift reflected in the orthant $\{ x \in \rr^n:\; x_1 \ge x_2 \ge \ldots \ge x_n  \}$. Two points are worth drawing attention to: reflected Brownian motion is a highly dependent system of processes; and, unlike a typical function concentration result, local times are not Lipschitz functions of the reflected process under the uniform norm. To the best of our knowledge there is no other way established in the literature to obtain concentration estimates of local times. 

The other compelling reason comes from mathematical finance. Given a financial market with a large (typically hundreds or thousands) number of stocks, one hedges risk by constructing a \textit{diversified portfolio}. Informally, this amounts to distributing the capital in holding and rebalancing shares over the entire equity market so that the value of the resulting portfolio is impervious to market risk. It is now clear that in mathematical terms this amounts to a concentration of the value process around a certain deterministic path. In Section \ref{sec:lip} we show through examples how such conclusions might be reached from our results on measure concentration. 

Before we proceed, let us add two caveats. One, we cannot compute expectations from concentration of measures, which require other methods. Two, because of the highly dependent structure in some of our examples our concentration bounds are not always Gaussian nor dimension independent. Gaussian tail estimate itself is a delicate property not shared by all stochastic processes. For example, consider the Bessel-square processes (see Revuz and Yor \cite{RY}). The marginal distribution of these diffusions are Chi-squares, which do not have sub-Gaussian tails. Thus, one cannot expect a Gaussian measure concentration property to hold for all Lipschitz functions of such processes. 

Our proofs depend on an original observation due to K. Marton \cite{M1, M2, M3}: Concentration of measure is a consequence of what are known as \textit{Transportation Cost Inequalities} (TCI). We explain this wonderful method in Section \ref{prelim:transport}. In fact what we prove in the text is, under suitable assumptions, multidimensional diffusion measures satisfy Quadratic Transportation Cost Inequality (QTCI) (see also, Talagrand \cite{T96b}, Dembo \cite{D}, Dembo and Zeitouni \cite{DZ}). Several other recent articles have taken a similar approach for proving concentration estimates for diffusions although  ours is the first proof of QTCI for diffusions and other semimartingales w.r.t.~the uniform metric. QTCI is unique in its advantages and is related to the log-Sobolev inequality, hypercontractivity, Poincar\'e inequality, inf-convolution, and Hamilton-Jacobi equations. For details, please consult Otto and Vilani \cite{OV}, Bobkov and G\"otze \cite{BG}, and Bobkov, Gentil, and Ledoux \cite{BGL}. Also see the recent articles by Gozlan \cite{G1,G2} and by Gozlan et al. \cite{GRS} which shows equivalence of QTCI with a restricted logarithmic Sobolev inequalities. The last article also gives the first proof that QTCI is preserved under bounded perturbation. 

Let us provide a brief review of literature of measure concentration in path space. Houdr\'e and Privault \cite{HP} and Nourdin and Viens \cite{NV} use tools from Malliavin Calculus to derive concentration inequalities for functionals on the Wiener space among other things. Proving TCI on the Wiener space using Girsanov theorem, as we have done, first appeared in Feyel and Ustunel \cite{FU}. Djellout, Guillin, and Wu \cite{DGW} provide characterization for $\mathbf{L}^1$-TCI for diffusions.  They also prove QTCI for diffusions with respect to the Cameron-Martin $\mathbf{L}^2$-metric. Several articles in analysis and geometry are also devoted to this topic. Fang and Shao \cite{FS05, FS07} consider TCI on abstract path spaces on connected Lie groups, Fang, Wang, and Wu \cite{FWW} consider TCI w.r.t.~the uniform metric for special diffusions on a complete Riemannian manifold; Gourcy and Wu \cite{GW} considers log-Sobolev inequalities under the $\mathbf{L}^2$-metric; Wang \cite{W02, W08}, studies generalized TCI on complete connected Riemannian manifolds; and Wu and Zhang \cite{WZ} prove QTCI for the uniform metric under an $\mathbf{L}^2$-contraction property of the diffusion semigroup. 

The outline of the article is as follows. In the following section we explain the connection between TCI and measure concentration. In Section \ref{sec:concendiff} we prove QTCI for semimartingale strong solutions of SDE's under suitable assumptions in one dimension. By using tensorization properties of QTCI, in Theorem \ref{thm:ndim} we extend concentration results to multidimensional processes with independent coordinates. However, independent coordinates are of limited appeal in applications.  In Theorem \ref{thm:change2} we use a perturbation argument to prove measure concentration properties for dependent processes that are locally absolutely continuous with respect to some multidimensional process with independent coordinates. The strenth of the concentration depends on a Birnbaum-Orlicz norm of the Radon-Nikod\'ym derivative of one measure with respect to the other. In Section \ref{sec:lip} we discuss concentration of Lipschitz functions of diffusions. This includes concentration of regular and stochastic integrals with respect to diffusions. 

Our main example is worked out in Section \ref{sec:rankbased} where we work out concentration estimates of local times for Brownian motion interacting through their ranks. These non-trivial processes can be described as follows. Let $\delta_1,\delta_2,\ldots,\delta_n$ be $n$ real constants. Consider the following system of stochastic differential equations: 
\begin{equation}\label{ranksde}
dX_i(t) = \sum_{j=1}^n \delta_j 1\left(X_i(t)=X_{(j)}(t)\right)dt + dW_i(t),\quad i=1,2,\ldots,n.
\end{equation}
Here $X_{(1)}(t) \ge X_{(2)}(t) \ge \ldots \ge X_{(n)}(t)$ are the coordinates of the process in the decreasing order, and $W=(W_1,W_2,\ldots,W_n)$ is an $n$-dimensional Brownian motion. The SDE models the movement of $n$ particles as interacting Brownian motions such that at every time point, if we order the positions of the particles, then the $i$th ranked particle from the top gets a drift $\delta_i$ for $i=1,2,\ldots,n$. As time evolves, the Brownian motions switch ranks and drifts, and hence their motion is determined by such time dependent interactions. Such processes have been considered in several recent articles. Among the more recent ones, see Banner, Fernholz, and Karatzas \cite{atlasmodel}, Banner and Ghomrasni \cite{BG}, McKean and Shepp \cite{sheppmckean}, Pal and Pitman \cite{palpitman}, Jourdain and Malrieu \cite{joumal}, Chatterjee and Pal \cite{chatpal,chatpal2}, Ichiba and Karatzas \cite{IK}, Ichiba et al.~\cite{IPBKF}, and Shkolnikov \cite{shkol2}. We refer the reader to the above articles for the list of applications of such models. They are similar in the discrete setting to the dynamic models of spin glasses studied by Arguin and Aizenman \cite{AA}, Ruzmaikina and Aizenman \cite{ruzaizenman}, Shkolnikov \cite{shkol}.

As an example of a typical result one can derive from this theory, let us state a theorem we prove in Section \ref{sec:rankbased}. 

\begin{thm}\label{thm:ltime}
Consider the model described in \eqref{ranksde}. Let $L_{j,j+1}(T)$ denote the local time at zero for the semimartingale $(X_{(j)}- X_{(j+1)})$ up to time $T$. For any choice of parameters $(\delta_1, \ldots, \delta_n)$ and constant initial points for the coordinate processes the random variable 
$\chi = \max_{1\le j \le n-1} L_{j,j+1}(T)$ satisfies the following tail estimate:
\[
P\left(  \abs{\chi - m_{\chi}} \ge rn^{5/2} \right) \le 2 \exp\left\{  - \frac{r^2}{CT}   \right\}, \quad r \ge 2\sqrt{2\log 2}.
\]
Here $m_{\chi}$ is the median of $\chi$ and $C$ refers to a universal constant.
\end{thm}

An improvement on this bound and other applications related to rank-based models have been done in the follow-up article by Pal and Shkolnikov \cite{PS}. Some of our results about TCI for multidimensional diffusions have been recently generalized in an article by \"Ust\"unel \cite{U}. 

\section{Preliminaries}

\subsection{Transportation cost and concentration}\label{prelim:transport}

Suppose $(\met, d)$ is a complete separable metric space equipped with the Borel sigma-algebra. For all probability measures $P$ and $Q$ on this probability space, consider the $p$-th Wasserstein distance
\[
\wass_p(P, Q)= \inf_{\pi} \left[ Ed\left( X, X' \right)^p \right]^{1/p},
\]
where the infimum is over all couplings of a pair of random elements $(X, X')$ such that the marginal law of $X$ is $P$ and that of $X'$ is $Q$. 

Now, we fix a particular probability $P$. Suppose there is a constant $C > 0$ such that for all probability measures $Q \ll P$ we have
\eq\label{qtci}
\wass_p(P,Q) \le \sqrt{2C\ent\left( Q \mid P \right) }.
\en
Here $\ent$ refers to the relative entropy $\ent\left( Q \mid P \right)= E^Q \log\left( {dQ}/{dP} \right)$. Then we say that $P$ satisfies the $\mathbf{L}^p$ Transportation Cost Inequality with the constant $C$. When $p=2$, this is often called a Quadratic Transportation Cost Inequality (QTCI). 

A function $f:\met \rightarrow \rr$ will be called Lipschitz if there is a positive constant $\alpha$ for which
\[
\abs{f(x)- f(y)} \le \alpha d(x,y), \quad x,y \in \met.
\]
The constant $\alpha$ is then referred to as the Lipschitz constant. We shall call a function to be $1$-Lipschitz if $\alpha$ can be taken to be one. Let $\lip$ denote the set of all $1$-Lipschitz functions on $(\met,d)$. The (very short) proof of the following theorem can be found in Ledoux \cite[p.~118]{L} and the original article by Marton \cite{M1}.

\begin{thm}\label{preq}
Suppose that $P$ satisfies QTCI with constant $C$. Then, one has the following concentration estimate for all $r \ge 2\sqrt{2C\log 2}$:
\begin{enumerate}
\item[(i)] For any measurable set $A$ such that $\mu(A) \ge 1/2$, one gets
\eq\label{aconcen}
\mu(A_r) \ge 1 - \exp\left\{ - r^2/ 8C \right\}, \quad A_r= \left\{ x\in \met: \; d(x,A)\le r   \right\}.
\en
\item[(ii)] And for any $f \in \lip$, one has
\eq\label{fconcen}
P\left( x:\; \abs{f(x) - m_f} \ge r \right) \le 2 e^{-r^2/8C},
\en
where $m_f$ is the median of $f$.
\end{enumerate}
\end{thm}

In general, any Wasserstein metric can be used to show \eqref{qtci}, however the choice of $p=2$ is important due to its tensorization property we describe below. For the proof see Ledoux \cite[p.~122-123]{L} (although our statement appears slightly different from Ledoux's monograph, they can be easily seen to be equivalent).

\begin{thm}\label{productgen}
Suppose $(\met_i, d_i, P_i)$, $i=1,2,\ldots,n$ be $n$ polish spaces with corresponding distances and probability measures on them. Consider the product metric space $\met^n=\met_1\times \ldots \times \met_n$ with the distance
\[
\bar d(x,y)= \sqrt{ \frac{1}{n}\sum_{i=1}^n d^2_i(x_i, y_i)}, \quad x,y\in \met^n,
\]
and the product probability measure on it $P=\otimes P_i$. Suppose that every $P_i$ satisfies the QTCI with the same constant $C$, then $P$ satisfies QTCI with constant $C/n$. 
\end{thm}

Finally, we need the following lemma which will be useful in the later text. Its (short) proof can be found in \cite{DGW}.

\begin{lemma}\label{lemmapushforward}[Lemma 2.1 in \cite{DGW}]
Suppose $\mu$ is a measure on a metric space $(E, d_E)$ that satisfies TCI with respect to the $\wass_p$ norm with a constant $C$. Let $(F, d_F)$ be another metric space. If the map $\Psi:(E,d_E) \rightarrow (F, d_F)$ is Lipschitz, i.e.,
\[
d_F\left( \Psi(x), \Psi(y) \right) \le \alpha d_E(x,y), \quad \text{for all}\quad x, y \in E,
\]
then $\tilde\mu=\mu\circ \Psi^{-1}$ satisfies TCI with the $\wass_p$ norm with a constant $C\alpha^2$ on $(F, d_F)$.
\end{lemma}

\section{Concentration of diffusion laws}\label{sec:concendiff}

Fix a finite positive time horizon $T$. Consider the metric space given by $\met=C[0,T]$, along with the norm
\eq\label{doned}
d(\omega,\omega')=\norm{\omega- \omega'}_{\infty}(T)=\sup_{0\le t \le T}\abs{\omega(t) - \omega'(t)}, \quad \omega\in C[0,T].
\en
Also consider the product space $\met^n=C^n[0,T]$. Let $\omega=(\omega_1, \ldots, \omega_n)$ and $\omega'=(\omega_1', \ldots, \omega'_n)$ be two elements in $\met^n$. Then
\eq\label{dnd}
\bar d(\omega, \omega')= \sqrt{\frac{1}{n} \sum_{i=1}^n \sup_{0\le t \le T}\abs{\omega_i(t) - \omega_i'(t)}^2 }= \sqrt{\frac{1}{n} \sum_{i=1}^n \norm{\omega_i - \omega_i'}_{\infty}^2(T) }.
\en

\subsection{Concentration for a fixed time horizon}

Consider a progressively measurable drift function $b(s, \omega)$ from $[0, \infty) \times C[0,\infty)$ into $\mathbb{R}$. That is to say, $b(s,\cdot)$ depends on the entire history of the process until time $s$. In particular, the processes we consider below need not be Markov. However, we do assume that the diffusion parameter $\sigma(t,X)$ depends only $t$ and $X(t)$. That is, $\sigma$ is a function from $[0,\infty) \times \rr$ into $\rr$. The following result is under the above set-up.

\begin{thm}\label{thm:onedim}
Fix a time point $T >0$. Suppose $X$ is a strong solution to the equation
\eq\label{mainsde}
dX(t) = b(t,X)dt + \sigma(t,X(t)) dW(t),\quad 0\le t \le T,\quad X(0)=x(0),
\en
where the coefficients satisfy the global Lipschitz conditions
\eq\label{condcoeff}
\begin{split}
\abs{b(t,\omega)- b(t,\omega')} &\le K_1 \sup_{0\le s \le t}\abs{\omega(s) - \omega'(s)} = K_1 \norm{\omega- \omega'}_{\infty}(t) \\
\text{and}\quad \abs{\sigma(t,x) - \sigma(t,y)}  &\le K_2\abs{x-y},\quad 0\le t \le T,\quad x, y\in \rr.
\end{split}
\en
Additionally assume that diffusion coefficient is bounded, i.e., $0\le \sigma(t,x) \le \kappa$. Here $K_1, K_2, \kappa$ are positive constants.

Let $P$ denote the law of $X$ considered as a probability measure on the metric space $(C[0,T], \norm{\cdot}_{\infty})$.  Then $P$ satisfies the quadratic transportation cost inequality \eqref{qtci} for the choice of 
\eq\label{whatisc}
C= 4\kappa^2Te^{4T(K_1^2T+4 K_2^2)}.
\en 
\end{thm}

\begin{proof} The set-up of the proof is very similar to the proof of Theorem 5.6 in \cite{DGW}. So we can conveniently skip some details. Consider the canonical sample space $\Omega= C[0,T]$ with a \textit{universal Brownian filtration} $\{ \fb_t, \;0\le t \le T \}$. The details of this construction can be found in Revuz and Yor \cite[Sec.~2 Chap.~3]{RY}. This amounts to a filtration that is generated by the coordinate process and suitably made right-continuous and \textit{augmented} with the null sets that are common to to the family of Wiener measures starting from any initial probability distribution. 

We consider the standard Wiener measure on this space, and thus the coordinate process, $W(t), \; 0\le t \le T$, is distributed as a standard Brownian motion. Since $X$ is a strong solution to the SDE \eqref{mainsde}, one can construct a copy of the process $X$ adapted to the above probability space that satisfies \eqref{mainsde} with respect to the coordinate Brownian motion $W$. Let $P$ denote the law of the process $X$ on $C[0,T]$.

Let $Q$ be any probability measure absolutely continuous with respect to $P$. Thus, if we define $M(T)$ to be the Radon-Nikod\'ym derivative of $Q$ with respect to $P$ it is a function of $X$. In particular, $M(T)$ is a measurable function with respect to the Brownian filtration constructed above. We now construct a martingale by defining
\eq\label{whatismt}
M(t)= E\left( M(T)\mid \fb_t   \right), \quad 0\le t \le T.
\en
Then $M(t)$ is a martingale with respect to the Brownian filtration. It follows then, see Revuz and Yor \cite[Sec.~3 Chap.~5]{RY}, that this martingale has a continuous version which can be written as a stochastic integral
\eq\label{whatismt2}
M(t) = 1 + \int_0^t H(s) dW(s),\quad 0\le t\le T,
\en 
for some predictable process $H$.

It suffices for QTCI to consider $Q$ such that the Radon-Nikod\'ym derivative is square-integrable under $P$. We will assume that the martingale $M$ is square-integrable for simplicity.

Now, since $M(t)$ is a continuous martingale, we use Girsanov's Theorem \cite[p.~327]{RY}. Thus, under the measure $Q$, the process $X$ satisfies the SDE
\eq\label{sde2}
\begin{split}
dX(t) &= b(t,X)dt + \sigma(t,X(t))\left[  dB(t) + M(t)^{-1}H(t) dt  \right]\\
&= \left[ b(t,X) + \sigma(t,X(t))M(t)^{-1}H(t) \right]dt + \sigma(t,X(t)) dB(t). 
\end{split}
\en
Here $B(t)$ is a Brownian motion under $Q$. Note that $M$ is never zero under $Q$ and hence $M^{-1}$ makes perfect sense. 

Now to bound the Wasserstein distance between the two measures $P$ and $Q$, we need to couple the solutions of the two SDE's \eqref{mainsde} and \eqref{sde2}. To do this, we construct a solution of \eqref{sde2} on a filtered probability space with the Brownian motion $B(t)$ running during time $[0,T]$, and use this same Brownian motion to create a strong solution of \eqref{mainsde}.

Thus we have a coupled process:
\[
\begin{split}
X^{(1)}(t)&= x_0 + \int_0^t \left[ b(s,X^{(1)}) + \sigma(s,X^{(1)}(s))M(s)^{-1}H(s) \right]ds + \int_0^t\sigma(s,X^{(1)}(s)) dB(s),\\
X^{(2)}(t)&= x_0 + \int_0^t b(s,X^{(2)}) ds + \int_0^t \sigma(s,X^{(2)}(s)) dB(s).
\end{split}
\]
We now estimate the uniform distance between these two processes.

Define a nondecreasing sequence of functions $\psi_n(x)$ which serve as an smooth approximation to the function $\abs{x}$ while satisfying 
\eq\label{whatispsin}
\abs{\psi_n'(x)} \le 1, \quad \text{and}\quad 0\le \psi_n''(x) \le \frac{2}{n K_2x^2}.
\en         
The details on how such a sequence can be constructed can be found in \cite[p.~291]{KS}.

Now define
\[
\begin{split}
\Delta(t)  := X^{(1)}(t) - X^{(2)}(t) &=\int_0^t \left[ b(s,X^{(1)}) - b(s,X^{(2)}) + \sigma(s,X^{(1)}(s))M(s)^{-1}H(s) \right]ds \\
&+ \int_0^t \left[ \sigma(s,X^{(1)}(s)) - \sigma(s,X^{(2)}(s))\right] dB(s).
\end{split}
\]      
Then by It\^o's rule
\eq\label{ito1}
\begin{split}
\psi_n(\Delta(t)) &=\int_0^t \psi_n'(\Delta(s))\left[  b(s,X^{(1)}) - b(s,X^{(2)}) \right]ds \\
&+\int_0^t \psi'_n(\Delta(s))\sigma(s,X^{(1)}(s))M(s)^{-1}H(s)ds\\
& +\frac{1}{2} \int_0^t \psi_n''(\Delta(s))\left[ \sigma(s,X^{(1)}(s)) - \sigma(s,X^{(2)}(s)) \right]^2ds\\
&+ \int_0^t \psi_n'(\Delta(s))\left[ \sigma(s,X^{(1)}(s)) - \sigma(s, X^{(2)}(s)) \right] dB(s).
\end{split}
\en    

Now, by the condition \eqref{condcoeff} on $\sigma$ and the property of the function $\psi''_n$ in \eqref{whatispsin} we get
\eq\label{ineq1}
0\le \frac{1}{2}\int_0^t  \psi_n''(\Delta(s)) \left[ \sigma(s,X^{(1)}(s)) - \sigma(s,X^{(2)}(s)) \right]^2ds\le \frac{t}{n}.
\en  
By using the Lipschitz property of the drift function $b$ and \eqref{whatispsin} we obtain
\eq\label{ineq2}
\abs{\int_0^t \psi_n'(\Delta(s))\left[  b(s,X^{(1)}(s)) - b(s,X^{(2)}(s)) \right]ds} \le K_1 \int_0^t \sup_{0\le u \le s}\abs{\Delta(u)}ds.
\en
Let $\xi(s)$ be the process $M(s)^{-1}H(s)$. Since $\sigma$ and $\psi_n'$ are bounded functions, we get
\eq\label{ineq3}
\abs{\int_0^t \psi'_n(\Delta(s))\sigma(s,X^{(1)}(s))M(s)^{-1}H(s)ds} \le  \kappa \int_0^t \abs{\xi(s)}ds.
\en
The final stochastic integral above in \eqref{ito1} is a local martingale. We use Doob's $\mathbf{L}^2$ inequality \cite[p.~54]{RY} to claim the following: 
\eq\label{ineq4}
\begin{split}
E &\sup_{0\le s\le t} \abs{\int_0^s \psi_n'(\Delta(u))\left[ \sigma(u,X(u)^{(1)}) - \sigma(u, X^{(2)}(u)) \right] dB(u)}^2\\
&\le 4E \int_0^t \left( \psi_n'(\Delta(s)) \right)^2 \left[ \sigma(s,X(s)^{(1)}) - \sigma(s, X^{(2)}(s)) \right]^2ds \le 4K_2^2 E\int_0^t \Delta^2(s) ds.
\end{split}
\en
The final inequality is due to the Lipschitz property of the coefficient $\sigma$ and \eqref{whatispsin}.

Combining the inequalities \eqref{ito1}, \eqref{ineq1}, \eqref{ineq2}, \eqref{ineq3}, \eqref{ineq4}, and applying Cauchy-Schwarz we get
\eq\label{ineq5}
\begin{split}
\frac{1}{4}E\sup_{0\le s\le t}\psi^2_n(\Delta(s))& \le K^2_1 E\left(\int_0^t \sup_{0\le u \le s}\abs{\Delta(u)}ds\right)^2\\
 &+ \kappa^2 E \left( \int_0^t \abs{\xi(s)} ds\right)^2  +  \frac{t^2}{n^2} + 4 K_2^2 E \int_0^t \Delta^2(s)ds \\
&\le K_1^2 t \int_0^t E\sup_{0\le u \le s}\abs{\Delta(u)}^2ds+ \kappa^2E \left( \int_0^t \abs{\xi(s)} ds\right)^2 \\
&+ \frac{t^2}{n^2} + 4K_2^2 \int_0^t E \sup_{0\le u \le s}\abs{\Delta(u)}^2 ds\\
&\le a_n + (K_1^2T+ 4 K_2^2) \int_0^t E\sup_{0\le u \le s}\abs{\Delta(u)}^2 ds, 
\end{split}
\en
where the constant $a_n$ is given by (again applying Cauchy-Schwarz)
\[
a_n= \kappa^2 T E\int_0^T {\xi^2(s)} ds + T^2/n^2.
\]

Now, one can construct $\psi_n$ such that $\psi_n(x)$ converges to $\abs{x}$ uniformly on compact sets. Recalling that the sample paths are continuous and taking the limit as $n$ goes to infinity in \eqref{ineq5}, we obtain
\[
E \sup_{0\le s\le t }{\Delta^2(t)} \le 4a + 4(K_1^2T+ 4K_2^2) \int_0^t E\sup_{0\le u \le s}\abs{\Delta(u)}^2 ds, \quad a=\kappa^2 T E\int_0^T {\xi^2(s)} ds.
\]
Let $\phi(t)$ denote the quantity $E \sup_{0\le s\le t }{\Delta^2(t)}$. Then it follows from above that
\[
\phi(t) \le 4a + 4(K_1^2T+ 4K_2^2 ) \int_0^t \phi(s)ds.
\]
By an application of Gronwall's lemma \cite[Sec.~1, Appendix]{RY}, we obtain $\phi(t) \le 4a e^{4(K_1^2T+4K_2^2)t}$. Thus we get
\[
E \sup_{0\le t \le T} \Delta(t)^2 \le 4\kappa^2Te^{4T(K_1^2T+4K_2^2)} E \int_0^T \xi^2(s) ds.
\]

Recall  that we are using the uniform distance $d(\omega,\omega')=\sup_{0\le t \le T}\abs{\omega(t) - \omega'(t)}$ between two paths $\omega, \omega'$ in $C[0,T]$. Thus, from our calculations above, we get
\eq\label{onedimeq1}
E d^2(X^{(1)}, X^{(2)}) \le    4\kappa^2Te^{4(K_1^2T+4K_2^2)T} E\int_0^T {\xi^2(s)} ds.
\en

On the other hand (see \cite[eqn.~5.7]{DGW}), the entropy of $Q$ with respect $P$ is given by
\eq\label{hqp}
\ent\left( Q \mid P  \right)= \frac{1}{2} E^Q \int_0^T \xi(u)^2 du.
\en

By combining the above inequality with \eqref{onedimeq1} we get
\[
E d^2(X^{(1)}, X^{(2)}) \le   8\kappa^2Te^{4T(K_1^2T+4K_2^2)} \ent\left( Q \mid P  \right),
\]
which completes the proof.

The final claim follows by the tensorization argument outlined in the previous section.
\end{proof}

\begin{thm}\label{thm:ndim}
For each $i=1,2,\ldots,n$, let $b_i(t, \omega)$ be progressively measurable real-valued drift function on $[0,\infty) \times C[0, \infty)$ and $\sigma_i(t,x)$ be measurable functions from $[0,\infty)\times \rr$ into $\rr$. Consider the following system of multidimensional stochastic differential equation:
\[
dX_i(t) = b_i(t,X_i) dt + \sigma_i(t, X_i(t)) dW_i(t), \quad 0\le t \le T, \quad X_i(0)=x_i.
\]
Assume that there are positive constants $K$ and $\kappa$ such that for every $i=1,2,\ldots,n$, the coefficients satisfy
\eq\label{condcoeff2}
\begin{split}
&\max_{1\le i\le n}\abs{b_i(t,\omega)- b_i(t,\omega')}  \le K\norm{\omega-\omega'}_{\infty}(t),\quad \omega, \omega' \in C[0,\infty)\\
&\max_{1\le i\le n, 0\le t \le T} \abs{\sigma_i(t,x) - \sigma_i(t,y)}\le K\abs{x-y},\quad x, y \in \rr\\
\text{and}&\qquad 0\le \min_{1\le i \le n}\sigma_i(t,x) \le \max_{1\le i \le n} \sigma_i(t,x) \le \kappa.
\end{split}
\en

Let $P$ denote the product law of independent processes $(X_1, \ldots, X_n)$. Then $P$, seen as a probability measure on the metric space $(C^n[0,1], \bar d)$, as in \eqref{dnd}, satisfies the QTCI \eqref{qtci} for the choice of $C= 4n^{-1}\kappa^2Te^{K^2T(T+4)}$. 
\end{thm}

In applications however the multidimensional diffusions with independent coordinates are of limited use, although it is common to use dependent diffusions that are a Girsanov change of measure of independent ones. If this change of measure is not too drastic, one should expect concentration properties to transfer to the dependent diffusion case. The following perturbation result makes this precise.  

We prove the result for a general metric space and any TCI. However, for our purpose in this paper the metric space will be the path space and the metric will be the uniform metric.

\begin{thm}\label{thm:change2}
Suppose $P$ and $R$ are mutually absolutely continuous probability measures on a complete separable metric space $(\met, d)$. Let $L$ be the logarithm of the Radon-Nikod\'ym derivative process of $R$ with respect to $P$. 
Suppose that $P$ satisfies $\mathbf{L}^p$ TCI ($p\ge 1$) with a constant $C$. 
Consider any $A$ such that $R(A) \ge 1/2$. For all $r$ such that
\[
r \ge 2\sqrt{ 2C\log 2 + 4C\norm{L}_1},
\]
one has
\eq\label{concenyoung}
1 - R(A_r) \le \exp\left(  - \frac{r^2}{8C\left(1+ 4\norm{L}_{\Phi}\right)}  \right).
\en
Here, $\norm{L}_1$ is the expectation w.r.t. $R$, and $\norm{L}_{\Phi}$ is the Birnbaum-Orlicz norm (w.r.t. $R$) defined by
\eq\label{BOnorm}
\norm{L}_{\Phi}:= \inf\left\{ a>0:\;  E^R \Phi(\abs{L}/a) \le 1 \right\}, \quad \Phi(t)=e^t-t-1.
\en
\end{thm}

\begin{proof} Let $A$ be a measurable subset and let $A_r$ be as described in \eqref{aconcen}. For any measurable subset $B$, let $\nu_B$ denote the probability measure $R$, conditioned on $B$, i.e.,
\[
\nu_B(\cdot) = \frac{R(B \cap \cdot)}{R(B)}.
\]
The measure $\nu_B$ is clearly dominated by $R$, and hence by $P$, due to the assumed mutual absolute continuity. 

Consider the Wasserstein distance between $\nu_A$ and $\nu_{B}$, where $B$ is the complement of $A_r$. By the triangle inequality and the fact that $P$ satisfies TCI, we get
\eq\label{compareh0}
\wass_p(\nu_A, \nu_{B}) \le \wass_p(\nu_A, P) + \wass_p(\nu_{B}, P) \le \sqrt{2C H\left( \nu_A \mid P \right)} + \sqrt{2C H\left( \nu_{B} \mid P \right)}.
\en

Now, since $d(A,B)\ge r$, for any coupling between $X \sim \nu_A$ and $Y \sim \nu_{B}$, it follows immediately that $d(X, Y) \ge r$. Thus, it follows that the left side of the above inequality is at least as large as $r$. We estimate the right side below. 

Obviously
\[
H\left( \nu_A \mid P \right) = E^{\nu_A} \left( \log \frac{d\nu_A}{d P}  \right)= E^{\nu_A} \left( \log \frac{d\nu_A}{d R}  \right) + E^{\nu_A} \left( \log \frac{dR}{d P}  \right). 
\]
In other words
\eq\label{compareh1}
H\left( \nu_A \mid P \right) \le   H\left( \nu_A \mid R \right) + \max\left(0, E^{\nu_A} \left( \log \frac{dR}{d P}  \right)\right).
\en

Now, by assumption, $\log dR / dP= L$. 
Thus
\eq\label{compareh2}
\begin{split}
E^{\nu_A} \left( \log \frac{dR}{d P}  \right)&= \frac{1}{R(A)}E^R\left(  1\{A\} L \right) \le  \frac{1}{R(A)}E^R\left(  1\{A\} L \right).
\end{split}
\en

Select $A$ such that $R(A) \ge 1/2$. Then, we get
\eq\label{compareh3}
\frac{1}{R(A)}E^R\left(  1\{A\} L \right) \le  \left(R(A)\right)^{-1} E^R L=2\norm{L}_1.
\en

For the set $B$, we follow a similar line of argument except for the final estimate above. We use a pair of Young's function (i.e., convex conjugates, see Neveu \cite[Appendix, p.~210-213]{N}) $\Phi$ and $\Psi$ given by 
\[
\Phi(t)= e^t - t - 1, \qquad \Psi(v)=(1+v)\log (1+v) - v.
\]
Recall the definition of a Birnbaum-Orlicz norm for suitable random variables:
\[
\norm{X}_{\Phi}:= \inf\left\{ a>0:\;  E^R \Phi(\abs{X}/a) \le 1 \right\}, \; \norm{Y}_{\Psi}:= \inf\left\{ a>0:\;  E^R \Psi(\abs{Y}/a) \le 1 \right\}.
\]
We will use the following generalization of the H\"older's inequality for Young functions (see \cite[Appendix, p.~210-213]{N}):
\[
\frac{1}{R(B)}E^R\left(  1\{B\} L \right) \le 2\norm{L}_{\Phi} \norm{1\{B\}/ R(B)}_{\Psi}.
\]

Let us estimate $ \norm{1\{B\}/ R(B)}_{\Psi}$. For any $a>0$, we get
\eq\label{young1}
\begin{split}
E^R \Psi\left[ \frac{1\{B\}}{aR(B)} \right]&= E^R\left\{  \left(1+\frac{1\{B\}}{aR(B)} \right)\log \left(1+\frac{1\{B\}}{aR(B)}  \right) - \frac{1\{B\}}{aR(B)} \right\}\\
&= R(B)\left(1+\frac{1}{aR(B)} \right)\log \left(1+\frac{1}{aR(B)}  \right) - \frac{1}{a}.
\end{split}
\en
Note that $R(B) < 1/2$ since $R(A)\ge 1/2$. We claim that 
\eq\label{whatisa1}
\norm{1\{B\}/ R(B)}_{\Psi} \le a:=\frac{\log \frac{1}{R(B)}}{1 - R(B) \log \frac{1}{R(B)}}.
\en

It suffices to check that for this value of $a$, the expression in \eqref{young1} is smaller than one. To see this, note that, by our definition
\[
1 + \frac{1}{aR(B)} = \frac{1}{R(B) \log 1/R(B)}.
\]
Thus
\[
\begin{split}
E^R \Psi\left[ \frac{1\{B\}}{aR(B)} \right]&= R(B)\left(1+ \frac{1}{aR(B)}\right) \log\left(1+\frac{1}{aR(B)}  \right) - \frac{1}{a}\\
&= \frac{1}{\log (1/R(B))} \log \frac{1}{R(B) \log 1/R(B)} - \frac{1}{\log 1/R(B)} + R(B)\\
&= 1 + \frac{1}{\log (1/R(B))}\log \frac{1}{ \log 1/R(B)} - \frac{1}{\log 1/R(B)} + R(B)
\end{split}
\]

Let $1/x:= -\log R(B)$. Our claim will follow if we show the following to be negative:
\[
h(x)= x\log x - x + e^{-1/x}= \frac{1}{\log (1/R(B))}\log \frac{1}{ \log 1/R(B)} - \frac{1}{\log 1/R(B)} + R(B).
\]
Since $0\le R(B) \le 1/2$ it is enough to check in the interval $(0, 1/\log 2)$. We claim that in this interval the function $h$ is convex. To verify, note that
\[
\begin{split}
h'(x)&= \log x + x^{-2} e^{-1/x}, \quad h''(x)=\frac{1}{x} - \frac{2}{x^3} e^{-1/x} + \frac{1}{x^4} e^{-1/x}.
\end{split}
\]
Noting that $e^{1/x} \ge 1 + 1/x$, we get
\[
\begin{split}
h''(x) &= x^{-4}e^{-1/x}\left[ x^3e^{1/x} -2x + 1 \right]\ge x^{-4}e^{-1/x}\left[ x^3 + x^2 -2x + 1 \right]\\
&= x^{-4}e^{-1/x}\left[ x^3 + (x- 1)^2 \right] \ge 0.
\end{split}
\]
This shows that $h$ is convex. Thus, to show check for the negative sign of $h$ it is enough to check at the end points. Plainly $h(0+)=0$, and numerically $h(1/\log 2)\approx -0.41 < 0$. By convexity it now follows that that $h(x)$ is negative for $x\in(0,1/\log 2)$.

This proves the claim \eqref{whatisa1}. In fact, we will simplify our choice of $a$ slightly more by defining
\[
a= 2 \log \frac{1}{R(B)},
\]
which is larger than the choice in \eqref{whatisa1} since $R(B) \le 1/2$.

Combining our argument so far we obtain
\eq\label{compareh4}
\frac{1}{R(B)} E^R \left( 1\{B\}  L \right) \le 4\norm{L}_{\Phi} \log \frac{1}{R(B)}.
\en

Thus, combining the above with \eqref{compareh0}, \eqref{compareh1}, \eqref{compareh2}, \eqref{compareh3} we get
\[
\begin{split}
\frac{r}{\sqrt{2C}} &\le \sqrt{ H\left( \nu_A \mid R \right) + 2\norm{L}_1} + \sqrt{ H\left( \nu_B \mid R \right) +  4\norm{L}_{\Phi} \log \frac{1}{R(B)} }
\end{split}
\]

Note that
\[
H\left( \nu_A \mid R \right) =\log\frac{1}{R(A)} \le \log 2, \quad H\left( \nu_B \mid R \right) =\log\frac{1}{R(B)}.
\]
Thus
\[
\begin{split}
\frac{r}{\sqrt{2C}}&\le \sqrt{ \log 2 + 2\norm{L}_1}+ \sqrt{\left(1+ 4\norm{L}_{\Phi}\right) \log \frac{1}{R(B)}}.
\end{split}
\]

Note that $R(B)=1-R(A_r)$. Hence for all $r$ larger than 
\[
2\sqrt{ 2C\log 2 + 4C\norm{L}_1},
\]
one has (say)
\[
2C\left(1+ 4\norm{L}_{\Phi}\right) \log \frac{1}{1-R(A_r)} \ge \frac{r^2}{4}. 
\]
Or, by rearranging terms, we get
\[
1 - R(A_r) \le \exp\left(  - \frac{r^2}{8C\left(1+ 4\norm{L}_{\Phi}\right)}  \right).
\]
This proves the assertion.
\end{proof}

\comment{
\begin{cor}
Moreover, for all $1$-Lipschitz $f$ on $C[0,T]^n$ into $\rr$, one gets the following estimate
\eq\label{approxe}
\abs{ \int f dP - \int f dR} \le \sqrt{C E^R \iprod{L}(T)}. 
\en
\end{cor}

\begin{proof} For the second assertion of the proposition, note that if $P$ satisfies QTCI, it automatically satisfies TCI with respect to the $\wass_1$ metric with the same constant. That is
\eq\label{wass1}
\wass_1(R,P) \le \sqrt{2C H(R \mid P)}.
\en
We now use the Monge-Kantorovich-Rubinstein dual characterization (\cite[p.~120]{L}) of the Wasserstein distance:
\[
\wass_1(R,P)= \sup\left[  \int f dR - \int f dP  \right],
\] 
where the supremum is running over all $1$-Lipschitz functions on the underlying metric space, i.e., $C^n[0,T]$ with the uniform norm. By replacing $f$ by $-f$, which is also $1$-Lipschitz, one gets the absolute value in \eqref{approxe}.

Now, $M$ is the Radon-Nikod\'ym derivative of $R$ with respect to $P$. Thus, we can follow the same steps we used to get \eqref{hqp} to deduce that $H(R\mid P)=1/2E^R \iprod{L}(T)$. Using \eqref{wass1}, this completes the proof of the result.
\end{proof}
}

\noindent\textit{Remark.} To get a feeling for the Birnbaum-Orlicz norm $\norm{\cdot}_{\Phi}$ used in \eqref{BOnorm}, let us compute this norm for the case when $L(T) = B(T)$, a one-dimensional standard Brownian motion. For any $a > 0$, a quick calculation will show
\[
E\left(  e^{a^{-1} \abs{B(T)}} \right) - \frac{1}{a} E \abs{B(T)} -1 = 2 e^{ T/2a^2} \Phi\left(  \sqrt{T}/a \right) - \frac{1}{a} \sqrt{\frac{2T}{\pi}} - 1.
\]
If we take $a=\sqrt{T}$, the above expression reduces to $2\sqrt{e}\Phi(1)- \sqrt{2/\pi}-1$ which comes to about $0.976$. In other words $\norm{B(T)}_{\Phi}\approx \sqrt{T}$.

Finally note that so far we have assumed that the starting points of the processes are given constants. When we randomize the starting values, it is not obvious what happens to the Transportation Cost Inequality of the mixture. This general problem is studied under the rubric of dependent tensorization and TCI (specialized to Markov chains of size two). See \cite{M1} and \cite{DGW} for more details, in particular Marton's coupling for Markov chains.
 
However, if we are interested in only concentration of measures, certain bounds can be easily obtained.

\begin{lemma}
Suppose $\mu$ is a probability measure on a metric space $(E, d_E)$. Let $(F, d_F)$ be another metric space. Suppose there is a regular conditional probability $P_x, \; x \in E,$ which is a probability measure on the Borel $\sigma$-algebra of $F$. Assume that each $P_x$ in the support of $\mu$ satisfies QTCI with a constant $C$. 

For any Lipschitz function $f: F \rightarrow \rr$, let $m_f(x)$ denote the median of $f$ with respect to the probability measure $P_x$. Let $m$ denote any constant. Then, if $P$ denotes the randomized measure $\int P_x(\cdot) \mu(dx)$, one gets
\[
P\left(  \abs{F - m} > r  \right)\le 2 \exp\left\{-\frac{r^2}{32C} \right\} + \mu\left(x:\; \abs{m_f(x) - m} > r/2  \right), \quad r\ge 2\sqrt{2\log 2}.
\]
\end{lemma}

\begin{proof} Follows from the triangle inequality. 
\end{proof}

To specialize the above result to our set-up at hand, take $(E, d_E)$ to be $\rr^n$ under the Euclidean norm and $(F, d_F)$ to be $C^n[0,T]$ under the $\bar d$ norm. One can take $m$ to be either the expectation or the median of the numbers $m_f(x)$.
Much better bounds can be obtained if we know that the Markov semigroup of the diffusion has some contraction properties. See the analysis by Wu \& Zhang \cite{WZ} in this direction.

\subsection{Concentration for infinite time horizon} 
For any two paths $\omega_1$ and $\omega_2$ in $C[0,\infty)$, we denote the uniform metric on their restriction to $[0,n]$ by $d_n(\omega_1, \omega_2)= \sup_{0\le t \le n} \left( \abs{\omega_1(t)-\omega_2(t)} \right)$. 
Then consider the locally uniform metric
\eq\label{metricrho}
\rho(\omega_1, \omega_2)= \max_n \frac{c_n d_n(\omega_1, \omega_2)}{1+d_n(\omega_1, \omega_2)}
\en
for some sequence of positive numbers $c_n$ (to be specified later) such that $\lim_{n\rightarrow \infty} c_n=0$. It is well-known  that this metric makes the space $C[0,\infty)$ a complete separable metric space.

The concentration results of the last subsection can all be extended to this case, although presumably it is less useful since it is more difficult to check Lipschitzness of functions with respect to the local metric. We include a statement for mathematical completeness specialized to arbitrary finite stopping times.

Consider a stopping time $\tau$ with respect to the right continuous filtration on $C[0,\infty)$ that is continuous with respect to metric $\rho$. Denote by $C[0,\tau]$ the metric space of paths in $C[0,\infty)$ such that $\tau(\omega) < \infty$ and $\omega(t)= \omega_{\tau}$ for all $t\ge \tau$. Clearly $C[0,\tau]$ is a closed subset of a Polish space, and is hence Polish itself under $\rho$. 

\begin{thm}\label{qtcistop}
Consider the same one-dimensional process as in Theorem \ref{thm:onedim} (with $K_1=K_2=K$) stopped at a continuous stopping time $\tau$. Let $P$ denote the law of the stopped process. Then $P$ satisfies QTCI with the constant $\ctilde$ given by 
\[
\ctilde=  4 \kappa^2 \max_n c_n^2  n e^{4K^2(n + 4)}. 
\]
In particular, if we choose $c_n= n^{-1/2}\exp(-2 K^2(n+4))$ in \eqref{metricrho}, then one can take $\ctilde=4\kappa^2$.
\end{thm}

\begin{proof}
The proof is very similar to the proof of Theorem \ref{thm:onedim}. Consider the coupling of processes $X^{(1)}$ and $X^{(2)}$ using the same driving Brownian motion, and consider the stopped processes $X^{(1)}_{t\wedge \tau_1}$ and $X^{(2)}_{t\wedge \tau_2}$. Here $\tau_1$ and $\tau_2$ are copies of $\tau$ applied to paths of $X^{(1)}$ and $X^{(2)}$ respectively. Let $\Delta(t)=X^{(1)}_{t\wedge \tau_1}-X^{(2)}_{t\wedge \tau_2}$.

Notice that the argument in the proof of Theorem \ref{thm:onedim} goes through for stopping times until inequality \eqref{ineq5} which gets now modified to
\eq\label{ineq15}
\frac{1}{4} E \sup_{0\le s \le t} \psi_n^2(\Delta(s)) \le a_n(t) + K^2(t+ 4) \int_0^{t} E \Delta(s)^2 ds.
\en
Here 
\[
a_n(t) = \kappa^2 t E \int_0^{t\wedge \tau_1} \xi(s)^2 ds + \frac{t^2}{n^2}.
\]

Thus, as before, by Gronwall's lemma we get
\eq\label{localbnd}
E d^2_n(X^{(1)}, X^{(2)})=E \sup_{0\le s \le n} \Delta(s)^2 \le 4 \kappa^2 n e^{4K^2(n + 4)} E \int_0^{n\wedge \tau_1} \xi(s)^2 ds.
\en

Now,
\[
\begin{split}
E \rho^2(X^{(1)}, X^{(2)})&\le \max_n c_n^2 E\left[ \frac{d^2_n(X^{(1)}, X^{(2)})}{1+d^2_n(X^{(1)}, X^{(2)})}\right]\\
&\le 4 \kappa^2 \max_n c_n^2  n e^{4K^2(n + 4)} E \int_0^{t\wedge \tau_1} \xi(s)^2 ds\\
&\le \ctilde E \int_0^{\tau_1} \xi(s)^2 ds.
\end{split}
\]
The rest of the proof is similar to the proof of Theorem \ref{thm:onedim}.
\end{proof}

\subsection{Classes of lipschitz functions}\label{sec:lip} 

In this section our objective is to work out a list natural examples of functions on the path space that are Lipschitz with respect to the uniform norm. Our aim is to show that paths of random processes derived from multidimensional diffusions lie in a cylinder around its ``expected path'' with exponentially high probability. Toward that aim, under suitable assumptions, we show concentration of processes of the type $\int_0^t \pi(u) du$ where $\pi$ is an adapted process, and of adapted local martingales.

We specialize to the case of $T=1$ and $1$-Lipschitz functions. Any other value of $T$ or of Lipschitz constant can be reduced to this case by scaling space and time.

\begin{lemma}\label{classlip}
Suppose $\{f(t), \; 0\le t \le 1\}$ be a collection of functions $f(t): C^n[0,1] \rightarrow \rr$ which are $1$-Lipschitz with respect to $\bar d$. That is, if $\omega$ and $\omega'$ are elements in $C^n[0,1]$, then 
\[
\abs{f(t)(\omega) - f(t)(\omega')}^2\le \frac{1}{n}\sum_{i=1}^n \norm{\omega_i - \omega'_i}^2_{\infty}, \quad \text{for every}\quad 0\le t\le 1 .
\]
Then the following functions are also $1$-Lipschitz.
\begin{enumerate}
\item[(i)] $\sup_t f(t)$ when the supremum is measurable.
\item[(ii)] For any Lipschitz function $\phi:\rr \rightarrow \rr$, the composition $\phi\circ f(t)$. In particular, $-f(t)$, $\abs{f(t)}$, and $\abs{f(t) - a(t)}$, where $a(t)$ is any non-random function.
\item[(iii)] The functions $g(t) = \int_0^t f(u) du, \quad 0\le t \le 1$.
\end{enumerate}

Further, suppose $f(t)$ is not known to be a priori Lipschitz. Let $\omega=(\omega_1, \ldots, \omega_n)$, where each $\omega_i \in C[0,1]$. Let $f^i(t)$ denote the function $f(t)$ as a function of $\omega_i$, while the rest of the coordinates are kept constant. Then, if for every choice of $i$ and $\omega_j, \; j\neq i$, the function $f^i(t)$ is Lipschitz in $\omega_i$ with coefficient $1/n$, i.e., 
\[
\abs{f^i(t)(\omega_i) - f^i(t)(\omega_i')}\le \frac{1}{n}\norm{\omega_i-\omega'_i}_{\infty},
\]
then $f(t)$ is $1$-Lipschitz with respect to the $\bar d$ norm.
\end{lemma}

\begin{proof}
The proofs of (i), (ii), and (iii) are obvious. To see the second part, choose a pair $\omega, \omega'$ in the product space. Construct a sequence of vectors $\eta(1), \ldots, \eta(n+1)$ in $C^n[0,1]$ such that
\[
\eta_j(i)=
\begin{cases}
& \omega'_j,\quad \text{if}\quad j\le i-1,\\
& \omega_j,\quad \text{otherwise}. 
\end{cases}
\]
Thus $\eta(1)=\omega$ and $\eta(n+1)=\omega'$.

By the triangle inequality and the property of being separately Lipschitz in each coordinate, we get
\[
\begin{split}
\abs{f(t)(\omega)- f(t)(\omega')}^2&\le n\sum_{i=1}^n \abs{f(t)(\eta(i+1)) - f(t)(\eta(i))}^2 \\
&\le{n} \sum_{i=1}^n \frac{1}{n^2}\norm{\omega_i-\omega_i'}^2_{\infty}=\frac{1}{n}\sum_{i=1}^n \norm{\omega_i-\omega_i'}^2_\infty.
\end{split}
\]
This shows that $f(t)$ is $1$-Lipschitz with respect to the $\bar d$ norm.
\end{proof}

We have the following corollaries. We focus on the set-up in Theorem \ref{thm:ndim}, although please keep in mind that for the following Gaussian concentration bounds and $\mathbf{L}^p$ TCI would suffice. In particular they hold for diffusions satisfying the conditions in Corollary 4.1 in \cite{DGW}.

\comment{
\begin{cor}[Concentration of the empirical process] Let $Q$ be a measure on $C[0,1]$ that satisfies the $\mathbf{L}^p$ TCI for some $p\ge 1$. Then every marginal distribution of the measure have Gaussian tails. Moreover, suppose $X(1), X(2), \ldots X(n)$ are iid processes with law $Q$. Let $P$ denote their joint probability measure. Then, for any Lipschitz function $f:\rr \rightarrow \rr$, we have
\[
P\left( \sup_{0\le t \le 1}\abs{ \frac{1}{n} \sum_{i=1}^n f\left(  X_i(t) \right) - \mu(t) } - \bar\mu > r \right) \le e^{- n r^2/ 8C}, \quad \text{for all}\; r \ge 2\sqrt{2C\log 2/n}. 
\] 
Here 
\[
\bar\mu=E \sup_{0\le t \le 1}\abs{ \frac{1}{n} \sum_{i=1}^n f\left(  X_i(t) \right) - \mu(t) }. 
\]
\end{cor}

\begin{proof} The Gaussian concentration of marginal distribution follows from considering the coordinate projection function.

Since $f$ is $1$-Lipschitz, it follows from Lemma \ref{classlip} that $n^{-1} \sum_{i=1}^n f\left(  X_i(t) \right)$ is $1$-Lipschitz with respect to the $\bar d$-norm. The rest of the operations, as shown in Lemma \ref{classlip}, preserve the property of being $1$-Lipschitz. The result now follows from Theorem \ref{thm:onedim} and Theorem \ref{preq}. The fact that we can work with expectations instead of medians is a standard fact that can be found in, e.g., \cite[Prop.~1.9, p.~11]{L}.
\end{proof}
}

\begin{cor}[Concentration of regular integrals]\label{concenreg}
Consider the set-up in Theorem \ref{thm:ndim}. Suppose that $\{ f(t),\; 0\le t \le 1 \}$ is a measurable real-valued process on $C^n[0,1]$ such that each $f(t)$ is $1$-Lipschitz with respect to the metric $\bar d$. Consider the process of integrals
\[
g(t)(\omega)=\int_0^t f(u)(\omega) du,\quad \text{and}\quad \mu(t) = \text{median of}\;\; g(t), \quad 0\le t \le 1.
\] 
Then, when \eqref{condcoeff2} holds, we have
\[
P\left( \abs{\sup_{0\le t \le 1}\abs{g(t) - \mu(t)}  - \bar\mu} \ge r  \right)\le 2 e^{-nr^2/8\sigma^2},
\]  
for all $r\ge 2\sigma\sqrt{2n^{-1}\log 2}$ where $\bar \mu$ is the median of $\sup_{0\le t \le 1}\abs{g(t) - \mu(t)}$ under the measure $P$.
\end{cor}

The proof is straightforward application of Lemma \ref{classlip} conclusion (iii). The supremum is measurable since $g(t)$ is always continuous. 

The last two corollaries are significant in mathematical finance where the value process of a portfolio is often expressed as a stochastic integral. In particular, it is a martingale under, what is known as, the \textit{risk-neutral} measure. The final value of such a martingale is often determined externally (i.e., pay-off from an European derivative). These corollaries together with Theorem \ref{thm:change2} can determine whether such portfolios can be impervious to random \textit{market risk}.

\begin{cor}[Concentration of martingales]\label{concenstoc}
Consider the set-up in Theorem \ref{thm:ndim} (iii). Let $\{N(t), \; 0\le t \le 1\}$ be a $P$-martingale such that $N_0=0$ and $N_1$ is a Lipschitz function  with respect to the metric $\bar d$. Then the following concentration inequality holds
\eq\label{eq:concenstoc}
P\left( \abs{\sup_{0\le t\le 1} \abs{N(t)} - \bar\nu} > r \right) \le 2 e^{-nr^2/8\sigma^2},
\en
for all $r\ge 2\sigma\sqrt{2n^{-1}\log 2}$, where $\bar \nu= E^P \sup_{0\le t \le 1}\abs{N(t)}$.
\end{cor}

\begin{proof} This is straightforward since $N_1$ is Lipschitz implies every other $N(t)$, which are conditional expectations of $N_1$, must also be Lipschitz.
\end{proof}

\subsection{Local times} Now we come to the discussion of a particularly important class of functions in the study of continuous stochastic processes: the local time. Consider a standard Brownian motion and let $L(t)$ denote its local time at zero. Then, by definition (see \cite[p.~227]{RY})
\[
L(t) = \lim_{\epsilon \rightarrow 0} \frac{1}{2\epsilon} \int_0^t 1\{ \abs{B(s)} < \epsilon \} ds.
\]
This is clearly not a Lipschitz function of the paths of $B(t)$. However, it is well-known that Brownian local time has Gaussian tails. In fact, a famous theorem of L\'evy states that $L(t)$ has the same law as $\abs{B(t)}$. One way to prove this concentration is by considering the Tanaka decomposition \cite[p.~239]{RY}:
\[
\abs{B(t)} = \int_0^t \text{sgn}(B(s))dB(s) + L(t).
\]
If $\beta$ denotes the process $\int_0^t \text{sgn}(B(s))dB(s)$, then L\'evy (and later, Skorokhod) showed that $L(t)= -\inf_{0\le s \le t} \beta(s)\wedge 0$ (see \cite[p.~239]{RY}). Thus $L(t)$ is a $1$-Lipschitz function of the paths of $\beta$. Since $\beta$ is another standard Brownian motion, it imparts Gaussian concentration to the local time function. 

This notion of obtaining local time as a Skorokhod map has been greatly generalized. For the rest of text we will focus on such generalizations. We refer the reader to the articles by Dupuis and Ramanan \cite{DR1, DR2} from which we borrow our description of the so-called Skorokhod Problem which we describe below. 

Consider a closed set $G \subseteq \rr^n$ and a set of unit vectors $d(x)$ for each point $x$ on the boundary of $G$ (say $\partial G$). Let $D^n[0, \infty)$ be the set of maps from $[0, \infty)$ to $\rr^n$ that are right continuous with left limits. For $\tilde\eta \in D^n[0, \infty)$ let $\abs{\tilde\eta}(T)$ be the total variation of $\tilde\eta$ on $[0,T]$ with respect to the Euclidean norm. 

\begin{defn}\label{skorokhod}
Let $\psi \in D^n[0,\infty)$ with $\psi_0 \in G$ be given. Then $(\phi, \tilde\eta)$ solves the Skorokhod Problem (SP) for $\psi$ with respect to $G$ and $d$ if $\phi_0=\psi_0$, and if for all $t \in [0, \infty)$
\begin{enumerate}
\item[(i)] $\phi(t) = \psi(t) + \tilde\eta(t)$;
\item[(ii)] $\phi(t) \in G$;
\item[(iii)] $\abs{\tilde\eta}(t) < \infty $;
\item[(iv)]$\abs{\tilde\eta}(t)= \int_0^t 1\left\{ \phi(s) \in \partial G \right\} d \abs{\tilde\eta}(s)$;
\item[(v)] There exists measurable $\gamma:[0,\infty) \rightarrow \rr^n$ such that $\gamma(s) \in d(\phi(s))$, the set of direction vectors at the point $\phi(s)$ ($\abs{\tilde\eta}$-almost surely), and 
\[
\tilde\eta(t) = \int_0^t \gamma(s)d \abs{\tilde\eta}(s).
\]
\end{enumerate}
\end{defn}

When, the solution of  SP exists (and is unique) for a large enough subset of the path space, the map that takes $\psi$ to $\phi$ is called the Skorokhod map. The Skorokhod map, following an original idea due to Skorokhod, is used to construct stochastic processes that are constrained to remain within $G$ by reflecting them inwards at the boundary $\partial G$ in the direction given by the vector field $d$.

A Skorokhod map is Lipschitz is there exists a positive constant $K$ such that if $(\phi, \tilde\eta)$ and $(\phi', \tilde\eta')$ are solutions to the SP for $\psi$ and $\psi'$ respectively, one has
\eq\label{splip}
\begin{split}
\sup_{t \ge 0} \norm{\tilde\eta(t) - \tilde\eta'(t)} &\le K \sup_{t\ge 0} \norm{\psi(t) - \psi'(t)},\\
\; \sup_{t \ge 0} \norm{\phi(t) - \phi'(t)} &\le K \sup_{t\ge 0} \norm{\psi(t) - \psi'(t)}.
\end{split}
\en
Here $\norm{\cdot}$ is the regular Euclidean norm. The constant $K$ is then called the Lipschitz constant of the Skorokhod map and $\psi$ is called the driving noise. Note that the norm used in \eqref{splip} is weaker than the $\bar d$ norm we have been using so far.

\begin{cor}\label{skoroconcen}
Suppose the Skorokhod map is $1$-Lipschitz. Then, if the noise process $\psi$ is chosen randomly with a law satisfying Theorem \ref{thm:ndim}, then constrained random path $\phi$ and the local time function $\eta$ also satisfies $QTCI$ with respect to the metric
\[
\hat d(\omega, \omega')=\sup_{0\le t \le T} \sqrt{\frac{1}{n}\sum_{i=1}^n \left( \omega_i(t) - \omega'_i(t) \right)^2}.
\]
As a corollary we obtain that reflected Brownian motion (RBM) satisfies QTCI with respect to the above norm whenever the above Skorokhod map is Lipschitz. 
\end{cor}

\begin{proof}
This is a corollary of the fact the QTCI is preserved under Lipschitz maps. See Lemma \ref{lemmapushforward}. The fact that $\hat d \le \bar d$ is straightforward.

To obtain RBM, one needs to take the driving noise as a typical path of multidimensional Brownian motion that satisfies QTCI by Theorem \ref{thm:ndim}.
\end{proof}

Note that the above result for RBM is not useful since we do not know the QTCI constant (which will depend on $n$). This is a rather non-trivial job and something that we work out in detail for a specific example in the following Section. However, several natural conditions that guarantee when the Skorokhod map is Lipschitz can be found in \cite{HR}, \cite[Thm.~2.2]{DI}, \cite[Thm.~3.2]{DR1} and \cite[Thm.~2.2]{DR2}. We will the following result from Dupuis and Ramanan \cite{DR2}. 

\begin{thm}[Theorem 2.2 in \cite{DR2}]\label{thm:DR}
Consider the Skorokhod problem on a polyhedral domain $G=\{ x\in \rr^n:\; \iprod{x, \eta_i} \ge c_i \}$ for some vectors $\eta_i$ and scalars $c_i$, where $i=1,2,\ldots,n$. Suppose the vector of direction of constraints (or, reflection) is constant on each face and are given by the vectors $d_i$, $i=1,2,\ldots,n$ which are linearly independent and satisfy $\iprod{d_i,\eta_i}=1$ for each $i$. Define the matrix
\eq\label{whatisqmatrix}
Q= [q_{ij}]= \begin{cases}\abs{\iprod{d_i, \eta_j}},&\quad \text{if} \quad i\neq j,\\
\abs{1-\iprod{d_i, \eta_i}}, & \quad \text{if}\quad i=j. 
\end{cases}
\en
If the spectral radius of $Q$ satisfies $\sigma(Q) < 1$, then the Skorokhod map is Lipschitz.
\end{thm}

\section{Applications to rank-based models}\label{sec:rankbased}

Consider the model described in \eqref{ranksde} in the Introduction. We call such a model as an $n$ particle rank-based model. For finite $n$, with arbitrary initial values of $X_i(0)$ and arbitrary drifts $\delta_i$, the existence and uniqueness in law of such an $n$ particle model is guaranteed by a standard application of Girsanov's Theorem. Please see Lemma 6 in \cite{palpitman}. A part of that lemma is reproduced below.

Let $\delta = (\delta_i, 1 \le i \le N) \in \rr^N$ and let $\mu$ be an arbitrary probability distribution
on $\rr^N$. Consider the canonical sample space and filtration for the multidimensional Brownian motion described in Section \ref{sec:concendiff}. Let $(X_1, \ldots, X_n)$ denote the coordinate map. Let $P^{\delta, \mu}$ denote the law of the $n$-particle rank-based model where the initial position is distributed as $\mu$. Thus, $P^{0,\mu}$ is the Wiener measure starting from $\mu$.

\begin{lemma}\label{lm2}
For each $t >0$ the law $P^{\delta,\mu}$ is absolutely continuous with respect to $P^{0,\mu}$ on $\fb_t$, with density
\eq
\label{girs} 
\exp \left( \sum_{j = 1}^N \delta_j  \beta_{j}(t) - \frac{t}{2} \sum_{j=1}^N \delta_j ^2 \right) 
\en
where $\beta_{j}$ can be defined by the expression
\eq\label{whatisbetaj}
\beta_j(t) = \sum_{i=1}^N\int_0^t 1\left\{\; X_i(s)=X_{(j)}(s)\right\}dX_i(s),\qquad 1\le j\le N. 
\en
Under $P^{\delta,\mu}$ the $\beta_{j}$'s are independent Brownian motions on $\rr$ with drift coefficients $\delta_j$ and diffusion coefficient $1$.
\end{lemma}

The other lemma we require considers the law of the ordered particle system under $P^{\delta, \mu}$:
\[
X_{(1)}(t) \ge X_{(2)}(t) \ge \ldots \ge X_{(n)}(t).
\]
This ordered system, as shown in \cite{atlasmodel} and \cite{palpitman}, is a reflected Brownian motion in the cone $\{ x\in \rr^n:\; x_1 \ge x_2 \ge \ldots \ge x_n  \}$. The following is in Lemma 4 of \cite{palpitman}, except that our ordering in \eqref{ranksde} is the reverse of the notation used in \cite{palpitman}.

\begin{lemma}\label{lm1} 
Let $X_i, 1 \le i \le n$ be a solution of the SDE \eqref{ranksde}, defined on the canonical space, for some arbitrary initial condition and arbitrary drifts $\{\delta_i\}$. Then for each $1 \le j \le n$ the $j$th ordered process $X_{(j)}$ is a continuous semimartingale 
with decomposition
\begin{equation}\label{tanak}
dX_{(j)}(t) = d \beta_{j}(t) + 
\frac{1}{\sqrt{2}} ( d L_{j,j+1}(t) -  d L_{j-1,j}(t)  )
\end{equation}
where the $\beta_{j}$'s for $1 \le j \le n$ are independent Brownian motions (with respect to the given filtration) with
unit variance coefficient and drift coefficients $\delta_j$, and are the same as appearing in \eqref{whatisbetaj}.

Moreover, $L_{0,1} = L_{n,n+1} =  0,$ and for $1 \le j \le n-1$ 
\eq
\label{occd}
L_{j,j+1}(t) = \lim_{ \epsilon \downarrow 0 } \frac{1}{ 2 \epsilon }  \int_0^t 1 ( (X_{(j)} (s)- X_{(j+1)}(s))/\sqrt{2} \le \epsilon ) ds, \quad t \ge 0,
\en
which is half the continuous increasing local time process at $0$ of the semimartingale 
$(X_{(j+1)} - X_{(j)})/\sqrt{2}$. Moreover,
the ordered system is a Brownian motion in the domain
\eq
\label{domn}
\left\{ x\in \rr^n:\; x_1 \ge x_2 \ge \ldots \ge x_n  \right\}
\en
with constant drift vector $(\delta_j, 1 \le j \le N)$ and normal reflection at each of the
$n-1$ boundary hyperplanes $\{x_{(i)} = x_{(i+1)}\}$ for $1 \le i \le n-1$.
\end{lemma}

\subsection{Concentration of intersection local times} 

Our objective is to show that the vector of boundary local times $(L_{j,j+1}, \; j=1,2,\ldots,n-1)$ satisfy the QTCI. We start with a lemma.

\begin{lemma}\label{lemma16}
Consider the rank-based model \eqref{ranksde}. Let $\overline{X}$ denote the center of mass process
\[
\overline{X}(t) = \frac{1}{n} \sum_{i=1}^n X_i(t), \quad t \ge 0.
\]
Then
\begin{enumerate} 
\item[(i)] $\xbar(t) - \xbar(0)$ is a Brownian motion with mean $\bar\delta=n^{-1}\sum_{i=1}^n \delta_i$ and diffusion coefficient $1/n$ and is independent of the vector of spacings $(X_{(i)}(t) - X_{(i+1)}(t),\; 1\le i \le n-1,\; t\ge 0 )$.
\item[(ii)] Let $\beta$ be an independent one-dimensional Brownian motion $(\beta_0=1)$ with a negative drift $-1$ which is reflected at the origin. Then the process defined by
\eq\label{whatisyi}
Y_i(t)= X_{(i)}(t) - \xbar(t) + \frac{\beta(t)}{\sqrt{n}}, \quad i=1,2,\ldots,n, \quad t\ge 0,
\en
is an $n$-dimensional Brownian motion with constant drift and identity covariance matrix which is normally reflected in the wedge $G= \cap_{i=1}^n G_i$, where
\eq\label{whatisg}
G_i=\{ x_i - x_{i+1} \ge 0 \},\; 1\le i \le n-1,\quad \text{and}\quad G_n=\left\{ \sum_i x_i \ge 0 \right\}.
\en
\end{enumerate}
\end{lemma}

\begin{proof}
The proof of the first assertions can be found in Lemma 7 in \cite{palpitman}. The argument for the second assertion uses Lemma \ref{lm1}. Please see \cite[p.~2188]{palpitman}.
\end{proof}

Notice that the spacing vector between the $Y_i$'s and that between the $X_i$'s are the same, i.e.,
\[
Y_i - Y_{i+1} \equiv X_{(i)} - X_{(i+1)}, \quad \text{for all}\quad i. 
\]
Thus, the local time at zero for every $Y_i-Y_{i+1}$ by $\sqrt{2}$ is exactly $L_{i,i+1}/2$, a fact that we use below. The advantage of considering $Y_i$'s is that now we can use Theorem \ref{thm:DR}.

\begin{thm}\label{localtimeqtci}
For any choice of parameters $(\delta_1, \ldots, \delta_n)$, the vector of increasing random processes 
\[
L_{j,j+1}(t), \quad j=1,2,\ldots,n-1, \quad 0\le t \le T,
\]
satisfies the QTCI with respect to the metric
\[
\hat d(\omega, \omega')=\sup_{0\le t \le T} \sqrt{\frac{1}{n}\sum_{i=1}^n \left( \omega_i(t) - \omega'_i(t) \right)^2}.
\]
with a constant $Cn^6T$ where $C$ is a universal constant.
\end{thm}
\bigskip

\noindent{Remark.} Let us point out to the reader that a better bound of $n^5T$ has been obtained in the follow-up article by Pal and Shkolnikov \cite{PS}. We strongly believe that this is optimal although we cannot prove it. 
\bigskip

\begin{proof}
Consider the Skorokhod map, described in Definition \ref{skorokhod}, on the polyhedral domain $G$ given in \eqref{whatisg}. As in Theorem \ref{thm:DR}, we can take the vectors $$\eta(i)= \frac{e(i)- e(i+1)}{\sqrt{2}},\quad \text{and}\quad d(i) = \eta(i),\quad \text{for each}\; i=1,2,\ldots, n-1.$$ Here the $e(i)$'s represent the standard basis vector in $\rr^n$. For $i=n$, we need to take $\eta(n)=n^{-1/2}\mathbf{1}$, and $d(n)=\eta(n)$, where $\mathbf{1}$ is the vector of all ones.

Then the $d(i)$'s are linearly independent and satisfies $\iprod{d(i), \eta(i)}=1$. Define the $n \times n$ matrix $D$ by 
\[
D = \left[  d(1) \mid d(2) \mid \cdots \mid d(n)  \right].
\]
Then $D$ has full rank. The matrix $Q$ in \eqref{whatisqmatrix} is then given 
\eq\label{Qrank}
Q= [q_{ij}]= I - D^*D=\begin{cases}  \frac{1}{2}1\left\{ \abs{i- j}=1 \right\},&\quad \text{if} \; 1\le i, j\le n-1, \\
0, & \quad \text{if}\quad i=j, \; i=n, \; \text{or}\; j=n. 
\end{cases}
\en
Here $D^*$ represents the matrix transpose of $D$.

It is easy to see that $Q$ is a submatrix of a stochastic transition matrix. That is, consider $P$ to be the transition probability matrix of a simple symmetric random walk on the integers $\{0,1,2,\ldots, n\}$, with absorbing boundary conditions. We then \textit{identify} states $0$ and $n$ by changing the $P$ matrix to have $P(n-1,0)=1/2=1-P(n-1,n)$ and $P(n,0)=1=1-P(n,n)$. 
Then $Q$ is the submatrix of transition probabilities corresponding to states $\{ 1,2,\ldots, n\}$. Since these states are transient, it follows that $\sigma(Q) < 1$. Thus, by Theorem \ref{thm:DR}, the Skorokhod map on this domain is Lipschitz. We need to estimate the Lipschitz coefficient. 

To do this, we need to find the set $B$ in {Assumption 2.1 (Set B)} in \cite[p.~160]{DR1}. We follow the notation used in \cite{DR1, DR2}. Let SP$(d(i), \eta(i), 0)$ denote the Skorokhod problem in our domain $G$. Consider another Skorokhod problem SP$(e(i), D^* \eta(i), 0)$, where $D^*$ refers to the transpose of $D$. As noted in \cite[p.~203]{DR2} (and easily verifiable), the matrix $Q$ corresponding to this problem is the same as the $Q$ in \eqref{Qrank}. Thus the Skorokhod map is again Lipschitz. 

Let $\hat B$ be the set satisfying (Assumption 2.1 \cite[p.~160]{DR1}) for the SP$(e(i), D^* \eta(i), 0)$ for some $\delta > 0$. Then, as mentioned in \cite[p.~202]{DR2}, $B=D \hat B$ satisfies Assumption 2.1 for the original problem SP$(d(i), \eta(i), 0)$.

We claim the following description of the set $\hat B$. Consider a vector $u$ such that
\eq\label{whatisu}
Qu < u, \qquad u_i > 0, \quad i=1,2,\ldots, n. 
\en
Then, we claim that
\eq\label{whatisbhat}
\hat B = \left\{x\in \rr^n:\; \abs{x_i} \le u_i  \right\}, \quad \delta= \min_i \left(  u_i - (Qu)_i \right). 
\en
Let us verify the conditions of Assumption 2.1 in \cite{DR1}. Consider a boundary point $z \in \partial \hat B$. Then, without loss of generality (and considering only points in the relative interior), we can assume that $z_i = u_i$, for some $i$, and $\abs{z_j} \le u_j$ for all $j \neq i$. In that case the inward normal to $z$ is $-e(i)$.

For our choice of $\delta$, assume that, for some $j$, one has 
\[
\abs{\iprod{z, D^* \eta(j)}} = \abs{\iprod{z, (I-Q) e(j) }} < \delta, \quad\text{since}\; \eta(j)=d(j)=De(j).
\]
Now, if $j=i$, then
\[
\abs{\iprod{z, (I-Q) e(j)}} = \abs{ {u_i} - \frac{z_{i+1} + z_{i-1}}{2}}  \ge {u_i} - \frac{u_{i+1} + u_{i-1}}{2} = u_i - (Qu)_i \ge \delta,
\]
since every $u_j$ is positive, $\abs{z_j} \le u_j$, and $Qu < u$. Thus, $\abs{\iprod{z, D^* \eta(j)}} < \delta$ implies $j \neq i$, and hence $\iprod{\nu, e(j)}=-\iprod{e(i), e(j)}=0$. This verifies \cite[eqn.~(2.2)]{DR1} and establishes our claim.

It is straightforward to see that the same $\delta$ works for the set $B=D\hat B$ as well. To wit, take any point $y\in \partial B$. Then, $y= Dz$, for some $z \in \hat B$. Thus $\iprod{y, \eta(j)}= \iprod{z, D^*\eta(j)} < \delta$ implies $-\iprod{\nu, e(j)}=0$, where $\nu$, as before, is the normal vector at $z$. But the normal vector at $y$, $\nu_y$, is obviously $(D^{-1})^*\nu$. Thus $\iprod{\nu_y, d(j)}=\iprod{\nu, D^{-1}d(j)}=\iprod{\nu, e(j)}=0$. This satisfies Assumption 2.1 in \cite{DR1}.

Thus to find the Lipschitz coefficient we need to find a vector $u$
satisfying \eqref{whatisu}. Let $v(x)$ be any nonnegative strictly concave
function on $[0,1]$ such that $v(0)=0$. Define $u(x)=v(x/n)$ for $0\le x\le n$, and let $u_k= v(k/n)$, for
$k=1,2,\ldots,n$. For example, we choose $v(x)=x(1-x)$. Then
\[
v''(x) = -2, \quad 0< x < 1.
\]
Then, by the
strict concavity of $v$ it follows that 
\[
u_k > \frac{1}{2}\left( u_{k-1} + u_{k+1}   \right), \quad
k=1,2,\ldots, n-1.
\]
Hence, it follows that $(Qu)_i < u_i$ for every
$i=1,2,\ldots,n-1$. For $i=n$ the inequality is not strict. However,
this is easy to correct since the row $Q_{n*}$ is the zero row and we
can choose $u_n$ to be any small enough positive number.

Now, for our choice of $v$, it follows that
\[
\delta= \min_i \left(  u_i - (Qu)_i \right) \ge \frac{1}{2} \inf_{0\le
x \le n} - u''(x) = \frac{1}{2n^2}\inf_{0\le x \le 1} -v''(x)= \frac{1}{n^2}. 
\]
In fact, to simplify, we define $\delta= n^{-2}$ for the rest
of the analysis.

Now to find what the Lipschitz constant is for the Skorokhod problem on the wedge $G$ we use the Remark made in page 161 in \cite{DR1}. If the set $B$ constructed above satisfies Assumption 2.1 in \cite{DR1} for some $\delta >0$, then one plus the diameter of $\delta^{-1} B$ serves as the Lipschitz constant for the Skorokhod map. We already know what $\delta$ is. Let us now estimate the diameter of the set $B$ in \eqref{whatisbhat}. Since extreme points are preserved under linear transforms, we get 
\[
\diam(B) \le 2 \max_{\sigma} \norm{\sum_{i=1}^n \sigma_i u_i d(i)},
\]
where the maximum is running over all choices of coefficients $\sigma_i = \pm 1$.

We claim that the above diameter is of the order $\sqrt{n}$. This is easy to see since each $d(i)$ has norm one and their inner products are given by the entries of the matrix $D^*D$. Thus by expanding the vector $\sum_{i=1}^n \sigma_i u_i d(i)$ for any choice of $\sigma$ we get
\[
\begin{split}
\norm{\sum_{i=1}^n \sigma_i u_i d(i)}^2&= \sum_{i=1}^n u_i^2 -  \sum_{i=1}^{n-1} \sigma_i \sigma_{i-1} u_i u_{i-1}\\
&\le \sum_{i=1}^n u_i^2+ \frac{1}{2}\sum_{i=1}^{n-1} \left(u_{i+1}^2 + u_i^2 \right) \le 3 \sum_{i=1}^n u_i^2.
\end{split}
\]

We now try to bound. By our choice of the functions $v$ and $u$ we get
\[
\begin{split}
\sum_{i=1}^n u^2_i= \sum_{i=1}^n v^2\left( \frac{i}{n} \right)\sim
n \int_0^1 x^2(1-x)^2dx= \frac{n}{30}. 
\end{split}
\]
In other words, 
\[
\diam(B) \le 2 \max_{\sigma} \norm{\sum_{i=1}^n \sigma_i u_i d(i)} \le 4\sqrt{n}.
\]

Combining with our estimate of $\delta$ we get that the diameter of $\delta^{-1}B$ is at most $4 n^{5/2}$. This, as argued, previously, serves as our Lipschitz constant for the Skorokhod map on $G$ as defined in \eqref{whatisg}. 

Now consider the RBM $Y(t)$ described in Lemma \ref{lemma16}. The $i$th coordinate process $Y_i(t)$ has a semimartingale decomposition 
\[
Y_i(t) = M_i(t) + \xi_i(t) + \tilde \eta_i(t),
\]
where $M_i$ is a martingale and $\xi_i(t)$, $\tilde \eta_i$ are of finite variation, the former being the absolutely continuous part (the drift) and the latter being the component that is mutually singular with respect to the Lebesgue measure (the local time). Comparing expressions \eqref{tanak} with \eqref{whatisyi} we get
\[
\tilde\eta_i(t) = \frac{1}{\sqrt{2}}\left( L_{i,i+1}(t) - L_{i-1,i}(t) \right)+ \frac{1}{\sqrt{n}}L_0(t), \quad L_{0,1}\equiv L_{n,n+1}\equiv0.
\]
Here $L_0(t)$ is the local time at zero for the Brownian motion $\beta(t)$. The drift is a constant $\delta_i-1/\sqrt{n}$, and the martingale $M_i$ is a Brownian motion, independent of all the other $M_j$'s.

Thus in Definition \ref{skorokhod} one can take the driving noise, $\psi$, to be a typical path of the $n$-dimensional Brownian motion with constant drifts $(M_i+ \xi_i, \; i=1,2,\ldots,n)$. Then, according to Theorem \ref{thm:ndim}, this multidimensional noise, during the time interval $0\le t \le T$, satisfies QTCI with a constant $4n^{-1}T$.

We will now use Lemma \ref{lemmapushforward}. We take the metric space $(E,d_E)$ to be $C^n[0,T]$ with the $\bar d$ metric, and take $\mu$ to be the multidimensional Wiener measure. We also take $(F,d_F)$ to be $C^{n}[0,T]$ with the metric 
\[
\hat d(\omega, \omega') =\sup_{0\le t \le T}\sqrt{\frac{1}{n}\sum_{i=1}^{n}(\omega_i(t)- \omega'_i(t))^2}.
\]
By \eqref{splip}, the map that takes $\psi$ to $\tilde\eta$ is Lipschitz. From what we have done so far, it is now clear that the map that takes $M_i+\xi_i$ to $\tilde\eta$ induces a Lipschitz map from $(E, d_E)$ to $(F, d_F)$ that is Lipschitz with a Lipschitz constant $O(n^{5/2})$. Since multidimensional Brownian motions satisfy QTCI with a constant $4n^{-1}T$, this shows that the process $\tilde \eta$ satisfies QTCI with a constant $O(n^{4})$.

Now consider the map that takes $\tilde\eta$ to the vector $L$. We claim that it is Lipschitz with a constant which is of order $\sqrt{n}$. To see this, note that the map that takes $L$ to $\tilde\eta$ is linear:
\[
\tilde\eta - \frac{L_0}{\sqrt{n}}\mathbf{1}= \frac{1}{\sqrt{2}}\left[ e(1)- e(2) \mid e(2)-e(3)\mid \ldots \mid e(n-1)-e(n) \right] L
\]
We claim that the smallest eigenvalue of the $(n-1)\times n$ dimensional matrix
\[
S=\frac{1}{\sqrt{2}}\left[ e(1)- e(2) \mid e(2)-e(3)\mid \ldots \mid e(n-1)-e(n) \right]. 
\]
is of the order $1/\sqrt{n}$.

But it is easy compute compute inner products between columns of $S$. Thus, for any vector $v\in \rr^{n-1}$, we get
\[
\norm{Sv}^2=\sum_{i=1}^{n-1} v_i^2 - \sum_{i=2}^{n-1} v_i v_{i-1} \ge \sum_{i=1}^{n-1} v_i^2 - \frac{1}{2}\sum_{i=2}^{n-1} \left(v^2_i + v^2_{i-1} \right)=\frac{v_1^2+v_{n-1}^2}{2}.
\]
The equality above is achieved when $v_1=v_2=\ldots=v_{n-1}$. Thus, under the constraint $\sum_i v_i^2=1$, the infimum of the above norm square is $1/n$. This shows that the map that takes $\tilde\eta$ to $L$ is Lipschitz with a constant $O(1/\sqrt{n})$.

Thus the map that takes $M+\xi$ to the vector of local times $L$ induces a Lipschitz map with a Lipschitz constant $O(n^{3})$.
By Lemma \ref{lemmapushforward} it follows then that the law of the process $(\eta_1, \ldots, \eta_{n-1})$ satisfies QTCI with a constant which is of the order of $n^{5}T$. This completes the proof of the Proposition.
\end{proof}

Finally we prove Theorem \ref{thm:ltime}.

\begin{proof}[Proof of Theorem \ref{thm:ltime}]
This proof is immediate once we note that the maximum function is Lipschitz with respect to the $\hat d$ norm required in Theorem \ref{localtimeqtci}.
\end{proof}

\section*{Acknowledgment} I thank Sourav Chatterjee and Michel Ledoux for many helpful discussions and information on the existing literature. I am grateful to an anonymous referee for a very helpful review of the preprint.

\bibliographystyle{alpha}

\begin{thebibliography}{50}

\bibitem{AA}
\textsc{Arguin, L.~-P.} and \textsc{Aizenman, M.} (2009). On the structure of quasi-stationary competing particles systems. \textit{The Annals of Probability} \textbf{37} 1080--1113.


\bibitem{atlasmodel}
\textsc{Banner, A.} and \textsc{Fernholz, R.} and \textsc{Karatzas, I.} (2005).
Atlas models of equity markets.
\textit{Ann. Appl. Probab.}, \textbf{15}(4) 2296--2330.


\bibitem{BG}
\textsc{Banner, A.} and \textsc{Ghomrasni, R.} (2008)
Local times of ranked continuous semimartingales. 
\textit{Stochastic Processes and their Applications} \textbf{118}, 1244--1253.

\bibitem{BGL}
\textsc{Bobkov, S.}, \textsc{Gentil, I.}, and \textsc{Ledoux, M.} (2001)
Hypercontractivity of Hamilton-Jacobi equations. \textit{J. Math. Pures Appl.} \textbf{80}, 669--696.


\bibitem{BG}
\textsc{Bobkov, S.} and \textsc{G\"otze, F.} (1999)
Exponential integrability and transportation cost related to logarithmic Sobolev inequalities. \textit{J. Funct. Anal.} \textbf{163}, 1--28.



\bibitem{chatpal}
\textsc{Chatterjee, S.} and \textsc{Pal, S} (2010)
\newblock A phase transition behavior for Brownian motions interacting through their ranks. 
\textit{Probability Theory and Related Fields} \textbf{147} (1-2), 123--159. 


\bibitem{chatpal2}
\textsc{Chatterjee, S.} and \textsc{Pal, S} (2008)
\newblock A combinatorial analysis of interacting diffusions. 
To appear in \textit{Journal of Theoretical Probability}.


\bibitem{D}
\textsc{Dembo, A.} (1997)
Information inequalities and concentration of measures. 
\textit{The Annals of Probability} \textbf{25}, 927--939.

\bibitem{DZ}
\textsc{Dembo, A.} and \textsc{Zeitouni, O.} (1996)
Transportation approach to some concentration inequalities in product spaces. 
\textit{Elec. Comm. in Probab.} \textbf{1}, 83--90.


\bibitem{DGW}
\textsc{Djellout, H.}, \textsc{Guillin, A.} and \textsc{Wu, L.} (2004)
Transportation cost-information inequalities and applications to random dynamical systems and diffusions. 
\textit{The Annals of Probability} \textbf{32} (3B), 2702 -- 2732.

\bibitem{DI}
\textsc{Dupuis, P.} and \textsc{Ishii, H.} (1991)
On Lipschitz continuity of the solution mapping to the Skorokhod problem. , with applications. 
\textit{Stochastics} \textbf{35}, 31--62.

\bibitem{DR1}
\textsc{Dupuis, P.} and \textsc{Ramanan, K.} (1999)
Convex Duality and the Skorokhod Problem. I. 
\textit{Probability Theory and Related Fields} \textbf{115}, 153--195.

\bibitem{DR2}
\textsc{Dupuis, P.} and \textsc{Ramanan, K.} (1999)
Convex Duality and the Skorokhod Problem. II. 
\textit{Probability Theory and Related Fields} \textbf{115}, 197--236.



\bibitem{FS05}
\textsc{Fang, S.} and \textsc{Shao, J.} (2005) Transportation cost inequalities on path and loop groups.
\textit{J. Funct. Anal.} \textbf{218} (2), 293--317.

\bibitem{FS07}
\textsc{Fang, S.} and \textsc{Shao, J.} (2007) Optimal transport maps for Monge-Kantorovich problem on loop
groups. \textit{J. Funct. Anal.} \textbf{248} (1), 225--257.




\bibitem{FWW}
\textsc{Fang, S.}, \textsc{Wang, F.~-Y.}, and \textsc{Wu, B.} (2008) Transportation-cost inequality on path spaces with
uniform  distance.  \textit{Stochastic Process. Appl.}  \textbf{118} (12), 2181--2197.


\bibitem{FK}
\textsc{Fernholz, R.} and \textsc{Karatzas, I.} (2009) Stochastic Portfolio Theory: A survey. In \textit{Handbook of Numerical Analysis: Mathematical Modeling and Numerical Methods in Finance}. Elsevier Publishing Company BV, Amsterdam, 89--168.


\bibitem{FU}
\textsc{Feyel, D.} and \textsc{Ustunel, A.~S.} (2004) The Monge-Kantorovitch problem and MOnge-Amp\`ere equation on Wiener space. \textit{Probability Theory and Related Fields.}



\bibitem{GW}
\textsc{Gourcy, M.} and \textsc{Wu, L.} (2006) Logarithmic Sobolev inequalities of diffusions for the $L^2$
metric. \textit{Potential Anal.}  \textbf{25} (1), 77--102.

\bibitem{G1}
\textsc{Gozlan, N.} (2007)
Characterization of Talagrand's like transportation-cost inequalities on the real line. 
\textit{Journal of Functional Analysis} \textbf{250}, 400--425. 

\bibitem{G2}
\textsc{Gozlan, N.} (2009)
A characterization for dimension free concentration in terms of transportation inequalities. 
\textit{The Annals of Probability} \textbf{37} (6), 2480--2498.



\bibitem{GRS}
\textsc{Gozlan, N.}, \textsc{Roberto, C.}, and \textsc{Samson, Paul-Marie} (2011)
A new characterization of Talagrand's transport-entropy inequalities and applications. 
\textit{The Annals of Probability} \textbf{39}(3), 857--880.


\bibitem{HR}
\textsc{Harrison, J.~M.} and \textsc{Reiman, M.~I.} (1981)
Reflected Brownian motion on an orthant. 
\textit{Annals of Probability} \textbf{9}, 302--308.


\bibitem{HP}
\textsc{Houdr\'e, C.} and \textsc{Privault, N.} (2002)
Concentration and deviation inequalities in infinite dimensions via covariance representations.
\textit{Bernoulli} \textbf{8} (6), 697--720.



\bibitem{IK}
\textsc{Ichiba, T.} and \textsc{Karatzas, I.} (2010)
Collisions of Brownian particles. To appear in \textit{The Annals of Applied Probability}.

\bibitem{IPBKF} 
\textsc{Ichiba, T.}, \textsc{Papathanakos, V.}, \textsc{Banner, A.}, \textsc{Karatzas, I.}, and \textsc{Fernholz, R.} (2010) Hybrid Atlas Models. To appear in \textit{The Annals of Applied Probability}. Preprint arXiv:0909.0065.

\bibitem{joumal}
\textsc{Jourdain, B.} and \textsc{Malrieu, F.} (2008).
\newblock Propagation of chaos and Poincar\'{e} inequalities for a system of particles interacting through their cdf.
\newblock \textit{Annals of Applied Probability} \textbf{18} (5) 1706--1736. 


\bibitem{KS}
\textsc{Karatzas, I.} and \textsc{Shreve, S.} (1991)
\textit{Brownian Motion and Stochastic Calculus. Second Edition.}
Graduate Texts in Mathematics, Springer.  

\bibitem{L}
\textsc{Ledoux, M.} (2001)
\textit{The Concentration of Measure Phenomenon.}
Mathematical Surveys and Monographs \textbf{89}. American Mathematical Society.


\bibitem{M1} 
\textsc{Marton, K.} (1996)
Bounding $\bar d$-distance by information divergence: a method to prove measure concentration. 
\textit{Annals of Probability} \textbf{24}, 857--866.

\bibitem{M2}
\textsc{Marton, K.} (1997)
A measure concentration inequality for contracting Markov chains.
\textit{Geometric and Functional Analysis} \textbf{6}, 556--571. 

\bibitem{M3}
\textsc{Marton, K.} (1998)
Mesure concentration for a class of random processes.
\textit{Probability Theory and Related Fields} \textbf{110}, 427--439.

\bibitem{sheppmckean}
\textsc{McKean, H. P.} and \textsc{Shepp, L.} (2005).
\newblock The advantage of capitalism vs. Socialism depends on the criterion.
\newblock Available at {\em www.emis.de/journals/ZPOMI/v328/p160.ps.gz}.


\bibitem{N}
\textsc{Neveu, J.} (1975)
\textit{Discrete parameter martingales.} North-Holland Mathematical Library.
North-Holland/ Elsevier Science. 


\bibitem{NV}
\textsc{Nourdin, I.} and \textsc{Viens, F.~G.} (2009)
Density formula and concentration inequalities with Malliavin calculus.
\textit{Electronic Journal of Probability} \textbf{14}, 2287--2309.



\bibitem{OV}
\textsc{Otto, F.} and \textsc{Villani, C.} (2000)
Generalization of an inequality by Talagrand and links with the logarithmic Sobolev inequality. \textit{J. Funct. Anal.} \textbf{173}, 361--400.



\bibitem{palpitman}
	\textsc{Pal, S.} and \textsc{Pitman, J.} (2008).
	\newblock  One-dimensional Brownian particle systems with rank dependent drifts.
	\newblock  \textit{Annals of Applied Probability} \textbf{18} (6), 2179--2207. 


\bibitem{PS}
\textsc{Pal, S.} and \textsc{Shkolnikov, M.} (2010)
\newblock Concentration of measure for systems of Brownian particles interacting through their ranks. 
\textit{Preprint.}



\bibitem{ruzaizenman}
\textsc{Ruzmaikina, A.} and \textsc{Aizenman, M.} (2005).
Characterization of invariant measures at the leading edge for competing particle systems.
\textit{The Annals of Probability}, \textbf{33} (1), 82--113.



\bibitem{RY}
\textsc{Revuz, D.} and \textsc{Yor, M.} (1999)
\newblock \textit{Continuous Martingales and Brownian Motion, Third Edition}.
\newblock A Series of Comprehensive Studies in Mathematics \textbf{293}, Springer.

\bibitem{shkol}
\textsc{Shkolnikov, M.} (2009). Competing Particle Systems Evolving by I.I.D. Increments. 
\textit{Electron. J. Probab.} \textbf{14}, 728--751. 


\bibitem{shkol2}
\textsc{Shkolnikov, M.} (2010).
Competing particle systems evolving by interacting Levy processes. 
\newblock \textit{Preprint}.


\bibitem{T91}
\textsc{Talagrand, M.} (1991) A new isoperimetric inequality for product measure, and the concentration of measure phenomenon. \textit{Israel Seminar (GAFA)}, Lecture Notes in Math. \textbf{1469}, 91--124. Springer-Verlag.

\bibitem{T94}
\textsc{Talagrand, M.} (1994) Sharper bounds for Gaussian and empirical processes. \textit{Ann. Probability} \textbf{22}, 28Ð76.

\bibitem{T95}
\textsc{Talagrand, M.} (1995) Concentration of measure and isoperimetric inequalities in product spaces. \textit{Publications Math\'ematiques de lÕI.H.E.S.} \textbf{81}, 73Ð205. 

\bibitem{T96}
\textsc{Talagrand, M.} (1996) A new look at independence. \textit{Ann. Probability}, \textbf{24}, 1Ð34 (1996). 


\bibitem{T96b}
\textsc{Talagrand, M.} (1996) Transportation cost for Gaussian and other product measures. \textit{Geometric and Funct. Anal.} \textbf{6}, 587--600.

\bibitem{T96c}
\textsc{Talagrand, M.} (1996) New concentration inequalities in product spaces. \textit{Invent. math.} \textbf{126}, 505--563.



\bibitem{U}
\textsc{\"Ust\"unel, A. S.} (2010). 
Transportation cost inequalities for diffusions under uniform distance. 
Preprint available at http://arxiv.org/abs/1009.5251.



\bibitem{W02}
\textsc{Wang, F.~-Y.} (2002) Transportation cost inequalities on path spaces over Riemannian manifolds.
\textit{Illinois J. Math.} \textbf{46}(4), 1197--1206.


\bibitem{W08}
\textsc{Wang, F.~-Y.} (2008) Generalized transportation-cost inequalities and applications. \textit{Potential Anal.} \textbf{28}
(4), 321--334.



\bibitem{WZ}
\textsc{Wu, L.} and \textsc{Zhang, Z.} (2004) Talagrand's $T_2$-transportation inequality w.r.t. a
uniform metric for diffusions.  \textit{Acta Math. Appl. Sin. Engl. Ser.}  \textbf{20} (3), 357--364.



\end{thebibliography}

\end{document}